\documentclass[12pt]{amsart}
 \usepackage{graphicx}

\usepackage{amsmath,graphics}
\usepackage{amsfonts,amssymb,xypic}

\theoremstyle{plain}
\newtheorem*{theorem*}{Theorem}
\newtheorem*{lemma*} {Lemma}
\newtheorem*{corollary*} {Corollary}
\newtheorem*{proposition*} {Proposition}
\newtheorem{theorem}{Theorem}[section]
\newtheorem{lemma}[theorem]{Lemma}
\newtheorem{corollary}[theorem]{Corollary}
\newtheorem{proposition}[theorem]{Proposition}

\theoremstyle{remark}
\newtheorem*{remark}{Remark}

\theoremstyle{definition}

\def \R {\Bbb{R}}
\def \Z {\Bbb{Z}}
\def \C {\Bbb{C}}

\def \II{\mathcal{I}}\def \NN{\mathcal{N}} \def \AA{\mathcal{A}}

\def \GG{\mathcal{G}}

\def\op{\operatorname}

\def\dt{\,dt}

\def\wti{\widetilde}

\def\eps{\epsilon}

\def\s{\sigma}

\def\Q{\Bbb{Q}}
\def\id{\mbox{id}}

\def\sign{\mbox{sign}}
\def\Z{\Bbb{Z}}
\def\C{\Bbb{C}}

\def\N{\Bbb{N}}
\def\l{\lambda}

\def\part{\partial}
\def\ll{\langle}
\def\rr{\rangle}

\def\a{\alpha}

\def\bp{\begin{pmatrix}}

\def\sm{\setminus}

\def\ep{\end{pmatrix}}
\def\bn{\begin{enumerate}}

\def\en{\end{enumerate}}
\def\ba{\begin{array}}
\def\ea{\end{array}}

\def\s{\sigma}

\def\a{\alpha}
\def\b{\beta}

\def\fr12{\frac{1}{2}}

\def\ker{\mbox{Ker}}
\def\be{\begin{equation} }
 \def\ee{\end{equation}}

\def\ord{\mbox{ord}}

\def\v{\varphi}

\def\G{\Gamma}

\def\trace{\mbox{trace}}

\def\genus{\mbox{genus}}

\def\zt{\Z[t^{\pm 1}]}

\def\cmtbf#1{} \def\cmt#1{}

\begin{document}

\title{Commensurability of knots and $L^2$--invariants}
\author{Stefan Friedl}
\address{Mathematisches Institut\\ Universit\"at zu K\"oln\\   Germany}
\email{sfriedl@gmail.com}

\date{\today}

\begin{abstract}
We show that the $L^2$-torsion and the von Neumann $\rho$-invariant  give rise to commensurability invariants of knots.
 \end{abstract}

\maketitle



\section{Introduction}
Two 3--manifolds $N_1$ and $N_2$ are called \emph{commensurable} if there exists
a diffeomorphism between some finite covers of $N_1$ and $N_2$.
We say that two oriented 3--manifolds
$N_1$ and $N_2$ are \emph{orientation--preserving commensurable} if there exists an orientation preserving
diffeomorphism between some finite covers of $N_1$ and $N_2$.

Given a knot $K\subset S^3$ we denote by $X(K)=S^3\sm \nu K$ the exterior of $K$. We endow $X(K)$ with the orientation inherited from the orientation of $S^3$.
We say that two knots $K_1$ and $K_2$ in $S^3$ are (orientation--preserving) commensurable,
if their exteriors  are (orientation--preserving) commensurable.

Note that $H_1(X(K);\Z)\cong \Z$. We denote by $\phi_K\colon \pi_1(X(K))\to \Z$ the unique (up to sign) epimorphism
and given $n\in \N$ we denote by $X(K)_n$ the unique $n$-fold cyclic cover of $X(K)$.
We say that two knots are \emph{cyclically (orientation--preserving) commensurable} if there exists an (orientation--preserving) diffeomorphism between cyclic covers of their exteriors.

Boileau, Boyer, Cebanu and Walsh \cite[Theorem~1.4]{BBCW10} recently proved the following theorem:

\begin{theorem}
Two hyperbolic knots without hidden symmetries which are commensurable are in fact cyclically commensurable.
\end{theorem}

We will not give a definition of `hidden symmetries' for a hyperbolic manifold but we point out that hyperbolic knots with hidden symmetries are `rare'. More precisely, the only hyperbolic knots which are known to have hidden symmetries are the Figure 8 knot and the two
dodecahedral knots of Aitchison and Rubinstein (see \cite{AR92} and Section \ref{section:ar92}).
 In fact it is conjectured (see \cite{NR92} and \cite[Conjecture~1.3]{BBCW10}) that these are the only hyperbolic knots with hidden symmetries. We refer to \cite{NR92} for some strong evidence towards this conjecture.
In light of the above we will restrict ourselves to the study of cyclic commensurability of knots. We refer to \cite{RW08}, \cite{GHH08}, \cite{Ho10} and \cite{BBCW10} for examples and background on cyclically commensurable knots.

In this note we will show that the Mahler measure $m(\Delta_K(t))$ of the Alexander polynomial $\Delta_K(t)$ of a knot and the integral over the Levine--Tristram signatures give obstructions to two knots being cyclically (orientation--preserving) commensurable.
Given a knot $K\subset S^3$ we define $\tau(K)\in \R$ to be the logarithm of the Mahler measure of the Alexander polynomial $\Delta_K(t)$, i.e.
\[\tau(K):=\ln(m(\Delta_K(t)))=\int_{z\in S^1}\ln|\Delta_K(z)|.\]
(Here and throughout this paper we equip $S^1$ with its Haar measure, i.e. the length is normalized to one.)
Note that if $A$ is a Seifert matrix for $K$, then an elementary calculation (see Section \ref{section:mahler}) shows that:
\[ \tau(K)=\int_{z\in S^1}\ln|\det(A(1-z)+A^t(1-\bar{z}))|.\]
Analogously, we define $\rho(K)$ as the integral of the Levine--Tristram signatures (see \cite{Le69,Tr69}), i.e.
\[  \rho(K) :=\int_{z\in S^1} \sign(A(1-z)+A^t(1-\bar{z})).\]
In Section \ref{section:mahler} we will see that both invariants can be computed easily from the zeros of the Alexander polynomial and finitely many calculations over the integers involving the Seifert matrix.

We say that a knot $K$ is \emph{admissible} if no root of unity is a zero of its Alexander polynomial.
It is well-known (see Lemma  \ref{lem:b1}) that a knot is admissible if and only if $b_1(X(K)_n)=1$ for any $n$.
It is straightforward to verify that if two admissible knots are cyclically commensurable, then they have the same
the genus and  the same degree of the Alexander polynomial and either both knots are fibered or neither is fibered (see Lemma \ref{mainlem} for details).
Our main theorem gives more subtle obstructions to knots being cyclically commensurable.

\begin{theorem}\label{mainthm}
Let $K_1$ and $K_2$ be two cyclically commensurable knots such that at least one of them is admissible.
Suppose that there exists a diffeomorphism
\[\varphi\colon X(K_1)_{n_1}\to X(K_2)_{n_2},\]
then there exist
 $m_1,m_2\in \Z$  such that
\[ \ba{rcl} n_1\cdot \tau(K_1)&=&n_2\cdot \tau(K_2), \\
n_1\cdot \rho(K_1)+m_1&=&\eps\cdot n_2\cdot \rho(K_2)+m_2,\ea \]
where $\eps=1$ if $\varphi$ is orientation preserving and $\eps=-1$ if $\varphi$ orientation reversing.
\end{theorem}

We immediately obtain the following corollary:

\begin{corollary}\label{maincor}
Let $K_1$ and $K_2$ be two cyclically commensurable knots such that at least one of them is admissible.
Then the following hold:
\bn
\item $\tau(K_1)$ and $\tau(K_2)$ are rationally dependent real numbers,
\item $\rho(K_1),\rho(K_2)$ and $1$ are rationally dependent real numbers.
\en
\end{corollary}

\begin{remark}
\bn
\item The invariants $\tau(K)$ and $\rho(K)$ can both be interpreted as $L^2$--invariants, we refer to Section \ref{section:l2} for details.
We will use the view point of $L^2$--invariants to prove our main theorem, this way the proofs of the two statements of Theorem \ref{mainthm} become very similar. In Section \ref{section:classical} we will give an alternative, more elementary, proof of the first statement of  Theorem \ref{mainthm}.
\item The results of Theorem \ref{mainthm} and Corollary \ref{maincor} also hold for knots in integral homology spheres.
The proofs are identical and for the sake of readability we only treat the case of knots in $S^3$.
\item
These results are rather different in spirit to most results in the literature on commensurability of knots and 3--manifolds. It is perhaps not entirely clear how these results can be used to address the most important questions in the study of commensurability. Perhaps one point of interest is that the $\rho$--invariant can be used to show that  certain knots which  are commensurable are not orientation--preserving commensurable.
We refer to Section \ref{section:example2} for such an example.
\item The fact that $L^2$--invariants give rise to commensurability invariants is not entirely surprising. Indeed,  the volume of a hyperbolic 3--manifold is multiplicative under finite covers, hence the quotient of volumes of commensurable hyperbolic 3--manifolds is rational; but the volume can also be interpreted in terms of $L^2$--torsion. We refer to Section \ref{section:morel2} and to \cite{LS99} and \cite[Section~4]{Lu02} for details.
\item It is in general an interesting and challenging problem to determine whether the numbers $\tau(K_1)$ and $\tau(K_2)$ (respectively $\rho(K_1), \rho(K_2)$ and $1$) are rationally independent.Mahler
Dubickas \cite[Corollary~3]{Du02} has shown that if $p(t)$ and $q(t)$ are the minimal polynomials of algebraic units (i.e. algebraic numbers whose inverses are also algebraic)
such that their Mahler measures are different, i.e. if $m(p(t))\ne m(q(t))$, then $m(p(t))/m(q(t))$ is irrational. In particular pairs of knots with such Alexander polynomials are not cyclically commensurable.
\en
\end{remark}

The paper is organized as follows: in Section \ref{section:l2} we will reinterpret $\tau(K)$ and $\rho(K)$ in terms of $L^2$--invariants and we will show that both invariants satisfy properties which we call `induction' and `restriction'. In Section \ref{section:proofs} we will use these two properties to prove Theorem \ref{mainthm}. Finally in Section \ref{section:examples} we study various examples.
In particular we will give an example of a pair of knots where the combination of $\tau(K)$ and $\rho(K)$ determines precisely which finite cyclic covers are diffeomorphic.\\

%
%

\noindent \textbf{Note.}
This note is a revised and shortened version of a longer unpublished manuscript by the author which is available from the author's webpage (see \cite{Fr11}).\\


\noindent \textbf{Conventions.} All manifolds are assumed compact, orientable and connected. All groups are assumed to
be countable.
\\

\noindent \textbf{Acknowledgments.} I would like to thank Steve Boyer, Radu Cebanu, Jonathan Hillman, Walter Neumann, Dan Silver, Peter Teichner, Stefano Vidussi and Genevieve Walsh for helpful conversations regarding this project. I am also grateful to the referee for carefully reading an earlier version of the paper and for suggesting the argument in Section \ref{section:powers}.

\section{$L^2$--invariants}\label{section:l2}

\subsection{The $L^2$--torsion}

Let $N$ be a compact 3--manifold with empty or toroidal boundary and let $\a\colon \pi_1(N)\to G$ be a non--trivial homomorphism to a group.
If the pair $(N,\a)$ is `det--$L^2$--acyclic', then by  \cite[Definition~3.4.1]{Lu02} the $L^2$--torsion $\tau(N,\a)\in \R\sm \{0\}$ is defined.
If $(N,\a)$ is not `det--$L^2$--acyclic', then we define $\tau(N,\a):=0$.

The precise definition of $\tau(N,\a)$ is irrelevant for the understanding of this paper.
The following well--known proposition summarizes the key properties of $\tau(N,\a)$ which we will use.

\begin{proposition}\label{prop:l2torsion}
Let $N$ be a compact 3--manifold with empty or toroidal boundary. Then the following hold:
\bn
\item Let  $\a\colon \pi_1(N)\to G$ be a homomorphism to a group and let  $\b\colon G\to H$ be a monomorphism of  groups, then
\[ \tau(N,\a\colon \pi_1(N)\to G)=\tau(N,\b\circ \a\colon \pi_1(N)\to H) \quad \mbox{ (Induction)}.\]
\item  Let $\a\colon \pi_1(N)\to G$  be a homomorphism to a group
 and let $\varphi\colon \pi_1(N)\to A$ be an epimorphism to a finite group $A$ which factors through $\a$. Denote the corresponding cover by $N_A$. Then
\[ \tau(N_A,\pi_1(N_A)\to \pi_1(N)\xrightarrow{\a}G )=|A|\cdot \tau(N,\a) \quad \mbox{ (Restriction)}.\]
\item Let $f\colon N\to M$ be a diffeomorphism and let $\a\colon \pi_1(M)\to G$ be a homomorphism, then
\[ \tau(N,\a\circ f_*)=\tau(M,\a).\]
\item Let $K\subset S^3$ be a knot, then
\[ \tau(X(K),\phi_K)=\tau(K).\]
\en
\end{proposition}

The first three properties follow immediately from \cite[Theorem~3.93~(5)~and~(6)]{Lu02} and from the definitions.
The last property follows from \cite[p.~206~and~Equation~(3.23)]{Lu02}, see also \cite[Equation~8.2]{LZ06}.

\subsection{The von Neumann $\rho$--invariant}

Let $N$ be a closed, oriented 3--manifold and let $\a\colon \pi_1(N)\to G$ be a group homomorphism.
Then there exists an invariant $\rho(N,\a)\in \R$, the von Neumann $\rho$--invariant  of the pair $(N,\a)$,
which was introduced by Cheeger and Gromov \cite{CG85}. We also refer to \cite[Section~2]{CT07} for a beautiful introduction and we refer to \cite[p.~313]{CW03} for an alternative definition.
An outline of the definition of the von Neumann $\rho$--invariant is given in the proof of Proposition \ref{prop:vonneumann}.

Before we state the key properties of $\rho(N,\a)$ we will introduce a few more objects.
Given a knot $K\subset S^3$ we denote by
$N(K)$ the 0--framed surgery on $K$. Note that the inclusion induces an isomorphism $\Z\cong H_1(X(K);\Z)\to H_1(N(K);\Z)$. We denote by $\phi_K\colon \pi_1(N(K))\to \Z$ the unique (up to sign) epimorphism.
Furthermore,  given $n\in \N$ we denote by $N(K)_n$
the unique $n$--fold cyclic cover of $N(K)$.
Let $A$ be a Seifert matrix for $K$. Then given $z\in S^1$ the Levine--Tristram signature (see
\cite[p.~242]{Le69} and \cite{Tr69}) of $K$ is defined as follows:
\[  \s(K,z) :=\sign(A(1-z)+ A^t(1-\bar{z})).
       \]
It is well--known that $\s(K,z)$ is independent of the choice of Seifert matrix.

The following proposition shows that the $\rho$--invariant has a striking similarity to the $\tau$--invariant.

\begin{proposition}\label{prop:vonneumann}
Let $N$ be a closed, oriented 3--manifold. Then the following hold:
\bn
\item Let  $\a\colon \pi_1(N)\to G$ be a homomorphism to a group and let  $\b\colon G\to H$ be a monomorphism, then
\[ \rho(N,\a\colon \pi_1(N)\to G)=\rho(N,\b\circ \a\colon \pi_1(N)\to H) \quad \mbox{ (Induction)}.\]
\item  Let $\a\colon \pi_1(N)\to G$  be a homomorphism
 and let $\varphi\colon \pi_1(N)\to A$ be an epimorphism to a finite group $A$ which factors through $\a$. Denote the corresponding cover by $N_A$. Then
\[ \rho(N_A,\pi_1(N_A)\to \pi_1(N)\xrightarrow{\a}G )=|A|\cdot\big( \rho(N,\a)-\rho(N,\varphi)\big) \quad \mbox{ (Restriction)}.\]
\item Let $f\colon N\to M$ be a diffeomorphism of oriented closed 3--manifolds and let $\a\colon \pi_1(M)\to G$ be a homomorphism, then
\[ \rho(N,\a\circ f_*)=\eps\cdot \rho(M,\a),\]
where $\eps=1$ if $f$ is orientation preserving and $\eps=-1$ otherwise.
\item Let $K\subset S^3$ be a knot, then
\[ \rho(N(K),\phi_K)=\rho(K).\]
If $\a\colon \pi_1(N(K))\to \Z/n$ is an epimorphism, then
\[ \rho(N(K),\a)=\frac{1}{n}\cdot \sum_{k=1}^{n}\s(K,e^{2\pi ik/n}).\]
\item  If $K^r$ is the mirror image of $K$, then $\rho(K^r)=-\rho(K)$.
\en
\end{proposition}

The statements of the proposition are well--known to the experts, but since we are not aware of a proof in the literature we will now provide a proof.

\begin{proof}
For the fourth statement we refer to  \cite[Proposition~5.1]{COT04} (cf. also \cite[Corollary~4.3]{Fr05}). The last statement follows either directly from (3)  and the observation that there exists an orientation reversing homeomorphism $N(K^r)\to N(K)$, or it follows from (4) and the observation that if $A$ is a Seifert matrix for $K$, then $-A$ is a Seifert matrix for $K^r$.

We now give a quick outline of a proof for the first three statements.
Let $g$ be a metric on $N$ and let $\a\colon \pi_1(N)\to G$ be a homomorphism. We think of all covers of $N$ as equipped with the pull--back metric which we will also denote by $g$.  Following \cite[Definition~(2.3)]{CT07} we define
\[ \eta(N,g):=\frac{1}{\sqrt{\pi}}\int_0^\infty t^{-1/2}\trace\left(De^{-tD^2}\right)\dt,\]
where $D$ denotes the signature operator on even-dimensional forms on $N$ given by $*d-d*$.
Note that the Hodge $*$--operator depends on the metric.
Similarly, following \cite[Definition~(2.8)]{CT07} we define
\[ \eta(N,g,\a):=\frac{1}{\sqrt{\pi}}\int_0^\infty t^{-1/2}\trace_G\left(\wti{D}e^{-t\wti{D}^2}\right)\dt,\]
where $\wti{D}$ denotes the signature operator on the even--dimensional forms on the $\a$--cover of $N$,
and the $G$--trace is defined in \cite[Section~2]{CT07}.
Then the von Neumann $\rho$-invariant is defined as follows:
\[ \rho(N,\a):=\eta(N,g)-\eta(N,g,\a).\]
The invariant $\rho(N,\a)$ is independent of the choice of the metric.
Note that the signature operator changes its sign when we change the orientation of $N$,
it follows that  $\eta(-N,g)=-\eta(N,g)$ and $\eta(-N,g,\a)=-\eta(N,g,\a)$.
Statement (3) is now an immediate consequence.

Now let  $\b\colon G\to H$ be a monomorphism. It follows from the definitions and
\cite[Section~1.1.5]{Lu02} that
\[ \eta(N,g,\b)=\eta(N,g,\a).\]
This immediately implies the first statement.

Finally
  let $\varphi\colon \pi_1(N)\to A$ be an epimorphism to a finite group $A$ which factors through $\a$. Denote the corresponding cover by $N_A$ and denote by $\a'$ the map $\pi_1(N_A)\to \pi_1(N)\xrightarrow{\a}G $.
It follows from the definitions and \cite[Theorem~1.9~(8)]{Lu02} that
\[ \ba{rcl} \eta(N_A,g,\a')&=&|A|\cdot \eta(N,g,\a)\\
\eta(N_A,g)&=&|A|\cdot \eta(N,g,\varphi).\ea \]
It now follows that
\[ \ba{rcl}
\rho(N_A,\a')&=&\eta(N_A,g,\a')-\eta(N_A,g)\\[0.1cm]
&=&|A|\cdot\big(\eta(N,g,\a)-\eta(N,g,\v)\big)\\[0.1cm]
&=&|A|\cdot\big((\eta(N,g,\a)-\eta(N,g))- (\eta(N,g,\v)-\eta(N,g))\big)\\[0.1cm]
&=&|A|\cdot\big(\rho(N,\a )-\rho(N,\v)\big).
\ea \]
\end{proof}

\subsection{The Mahler measure and the Levine--Tristram signatures}\label{section:mahler}

In this section we will recall the well--known fact that one can calculate $\tau(K)$ and $\rho(K)$ from the zeros of the Alexander polynomial and finitely many calculations over the integers.

First recall that for a polynomial $p(t)$ the Mahler measure is defined as
\[ m(p(t))=\exp\big( \int_{z\in S^1} \ln|p(z)|\big).\]
If $p(t)=c_nt^n+c_{n-1}t^{n-1}+\dots+c_1t+c_0$ (with $c_n\ne 0$) and if  $r_1,\dots,r_n$ are the roots of $p(t)$, then it follows from Jensen's formula that
\be\label{equ:jensen} m(p(t)) = |c_n| \cdot \prod\limits_{j=1}^n \mbox{max}(|r_j|,1).\ee
Note that for a polynomial $p(t)$ we have $m(p(t))=1$ if and only if $p(t)$ is the product of cyclotomic
polynomials (cf. \cite[Lemma~19.1]{Sc95}). It follows that Theorems \ref{mainthm} and \ref{mainthm2v2} give in most cases  non--trivial information on commensurability.
 We refer to \cite{SW04} for more details and references.

Now  let $A$ be a $k\times k$--Seifert matrix for $K$. By the multiplicativity of the Mahler measure and since $m(t^{-1}-1)=1$ we have
\[ \ba{rcl} m(\Delta_K(t))&=&m(\det(At-A^t))\\
&=&m(t^{-1}-1)^k\cdot m(\det(At-A^t))\\
&=&m((t^{-1}-1)^k\det(At-A^t))\\
&=&m(\det((t^{-1}-1)(At-A^t)))\\
&=&m(\det(A(1-t)+A^t(1-t^{-1}))).\ea \]
We therefore see that
\[  \tau(K)=\int_{z\in S^1} \ln|\det(A(1-z)+A^t(1-z^{-1}))|.\]

We now turn to the study of the $\rho$--invariant and the Levine--Tristram signatures.

\begin{proposition} \label{propu1sign}
Let $K$ be a knot and $z\in S^1$. Then the following hold:
\bn
 \item The function $z\mapsto \s(K,z)$ is
 locally constant outside of the zero set of $\Delta_K(t)$.
 \item If $\Delta_K(t)$ has no zeros on $S^1$, then $\rho(K)=0$.
 \item If $\Delta_K(t)$ has exactly two zeroes $z=e^{\pm 2\pi i s}, s\in (0,\frac{1}{2})$ on $S^1$, and if $w=e^{\pm 2\pi it},t\in (0,\frac{1}{2})$, then
 \[ \s(K,w)=\left\{ \ba{ll} 0, &\mbox{ if $t\in (0,s)$,} \\ \s(K),&\mbox{ if $t\in (s,\frac{1}{2})$,}\ea \right.\]
  where $\s(K):=\s(K,-1)$ denotes the signature of the knot $K$.
  In particular
 \[ \rho(K)=\big(1-2s\big)\s(K).\]
\en
\end{proposition}

\begin{proof}
 A standard
argument in linear algebra shows that the function $z \mapsto \s(K,z)$ is continuous on
\[ \{ z\in S^1\, \,| \,\,A(1-z)+A^t(1-\bar{z}) \mbox{ is non--singular} \}.\]
Using $\frac{1-z}{1-\bar{z}}=-z$ we see that this set equals
\[ \{1\} \cup \{ z\in S^1 \,\,| \,\,\Delta_K(z)=\det(Az-A^t)\ne 0 \}.\]
Levine \cite[p.~242]{Le69} showed that the signature function is continuous at $z=1$. The other statements now follow immediately from the definitions and the observation that $\s(K,1)=0$.
\end{proof}

 Note that it follows from the above considerations that $\tau(K)$ and $\rho(K)$ are in some sense complementary.
  The former invariant only depends on the zeros of $\Delta_K(t)$ outside of the unit disk, whereas the latter is related to the zeros of $\Delta_K(t)$ on the unit circle.

\section{Proof of the main results}\label{section:proofs}

\subsection{Finite cyclic covers} \label{section:cycliccover}

We recall the following well--known lemma:

\begin{lemma}  \label{lem:b1}
Let $K\subset S^3$ be an admissible knot, then  $b_1(X(K)_n)=1$ for any $n\in \N$.
\end{lemma}

\begin{proof}
It is well--known (see \cite[p.~17]{Go77}  for details) that
\[ H_1(X(K)_n)\cong \Z \oplus G\]
where $G$ is an abelian group which satisfies
 \be \label{equh1ln} |G| = \left| \prod_{j=1}^n \Delta_K(e^{2 \pi i j/n})\right|, \ee
where $0$ on the right hand side means that $G$ is infinite. It follows that $G$ is finite unless an $n$--th root of unity is a zero of $\Delta_K(t)$. The lemma is now an immediate consequence.
\end{proof}

Note that it
can be easily checked whether $\Delta_K(t)$ has a zero which is a root of unity. Indeed, if $\xi$
is a primitive $n$--th root of unity, then the minimal polynomial $p(t)$ of $\xi$ has degree $\varphi(n)$
where $\varphi$ is the Euler function. Since $p(t)$ divides $\Delta_K(t)$ we get an upper
bound on $\varphi(n)$ (and hence on $n$) in terms of $\deg(\Delta_K(t))$. In particular one can
show that $t^2-t+1$ (which incidentally is the Alexander polynomial of the trefoil knot) is the only Alexander polynomial of degree 2 which has zeroes which are roots
of unity.

\subsection{Invariants from the infinite cyclic cover}

The following lemma gives some fairly obvious invariants of cyclic commensurability.

\begin{lemma} \label{mainlem}
Let $K_1$ and $K_2$ be two cyclically commensurable knots such that at least one of them is admissible.
Then the following hold:
\bn
\item $\deg(\Delta_{K_1}(t))=\deg(\Delta_{K_2}(t)),$
\item $\Delta_{K_1}(t)$ is monic if and only if $\Delta_{K_2}(t)$ is monic,
\item $\genus(K_1)=\genus(K_2),$
\item  $K_1$ is fibered if and only if $K_2$ is fibered.
\en
\end{lemma}

The lemma is well--known to the experts but for completeness' sake we will give a proof.

\begin{proof}
Suppose that $K_1$ and $K_2$ are cyclically commensurable.
Then there exist $n_1,n_2\in \N$ such that $X(K_1)_{n_1}$ and $ X(K_2)_{n_2}$ are diffeomorphic.
We write $Y_i=X(K_i)_{n_i}$.
Since $K_1$ or $K_2$ is admissible it follows from Lemma \ref{lem:b1} that $b_1(Y_1)=b_1(Y_2)=1$.
 We denote the projection maps $Y_i\to X(K_i)$
by $p_i$ and for $i=1,2$ we pick epimorphisms ${\psi}_i\colon H_1(Y_i)\to\Z$ and $\phi_i\colon H_1(X(K_i))\to\Z$.

Let $i\in \{1,2\}$.
The map $\pi_1(Y_i)\to \pi_1(X(K_i))\xrightarrow{\phi_i}\Z$ defines an epimorphism $\pi_1(Y_i)\to n_i\Z$.
As mentioned above, there exists, up to sign, a unique epimorphism from $\pi_1(Y_i)$ onto a free cyclic group. It thus follows that
\[ \ker(\pi_1(Y_i)\xrightarrow{\psi_i} \Z)=\ker(\pi_1(Y_i)\to \pi_1(X(K_i))\xrightarrow{\phi_i}\Z)=\ker(\pi_1(X(K_i))\xrightarrow{\phi_i} \Z).\]
It thus follows that the unique infinite cyclic cover of $Y_i$ equals the infinite cyclic cover of $X(K_i)$.

Since $Y_1$ and $Y_2$ are diffeomorphic it follows that   $K_1$ and $K_2$ have diffeomorphic infinite cyclic covers. The first statement   now follows from
the fact that for any knot the degree of the Alexander polynomial equals the dimension of the rational homology of the infinite cyclic cover.
The second statement follows from the well--known fact that $\Delta_{K}(t)$ is monic if and only if the homology of the infinite cyclic cover is a free abelian group.
The last statement follows from the  fact (see \cite{St62}) that a knot is fibered if and only if the fundamental group of the infinite cyclic cover is finitely generated.

Finally we turn to the proof of the third statement.
 We now denote by $\psi_i,\phi_i, i=1,2$ the
elements in first cohomology groups of $Y_i, X(K_i)$ corresponding to the homomorphisms to $\Z$.
It  follows from the previous discussion that $(p_i)^*(\phi_i)=n_i{\psi_i}$.
We denote the Thurston norm of a 3--manifold $M$ by $x_M$ (cf. \cite{Th86} for details).
 Then we have
\[
x_{Y_i}(\psi_i)=\frac{1}{n_i}\cdot x_{{Y_i}}\big((p_i)^*(\phi_i)\big)=\frac{1}{n_i}\cdot n_i\cdot x_{X(K_i)}(\phi_i)=x_{X(K_i)}(\phi_i)=
\max\{0,2\,\genus(K_i)-1\},\] where the second equality is due to Gabai \cite{Ga83} and the other equalities follow from elementary properties of the Thurston norm (see \cite{Th86}).
\end{proof}

\begin{remark}
Let $K$ be a hyperbolic knot.
If $K$ is commensurable to a different knot $K'$, then it follows from
\cite[Theorem~1.7]{BBCW10} that $K$ and $K'$ are fibered. This can be viewed as a significant strengthening of the observations
of Lemma \ref{mainlem} under the somewhat stronger hypothesis that  $K$ is hyperbolic.
\end{remark}

\subsection{Proof of Theorem \ref{mainthm}}   \label{section:proofknot}

We will now prove the following theorem which can be viewed as a more refined version of Theorem \ref{mainthm}.

\begin{theorem} \label{mainthm2v2}
Let $K_1$ and $K_2$ be two knots and $n_1,n_2\in \N$ with $b_1(X(K_i)_{n_i})=1$. If  there exists
a diffeomorphism
\[ \varphi\colon  X(K_1)_{n_1}\xrightarrow{\cong} X(K_2)_{n_2},\] then
\[ n_1\cdot \tau(K_1)=n_2\cdot \tau(K_2)\in \R, \]
and
\[ n_1\cdot  \rho(K_1)-\sum_{k=1}^{n_1} \s(K_1,e^{2\pi ik/n_1})
=\eps \cdot \big( n_2\cdot \rho(K_2)-\sum_{k=1}^{n_2} \s(K_2,e^{2\pi ik/n_2})\big)\in \R,  \] where $\eps=1$
if $\varphi$ is orientation preserving, and $\eps=-1$ if $\varphi$ is orientation reversing.
\end{theorem}

Theorem \ref{mainthm} follows immediately from Lemma \ref{lem:b1} and from the observation that  Levine--Tristram signatures are integers.

\begin{proof}
We write $Y_i=X(K_i)_{n_i}$. We denote the projection maps $Y_i\to X(K_i)$
by $p_i$. For $i=1,2$ we pick epimorphisms ${\psi}_i\colon  H_1(Y_i)\to\Z$ and $\phi_i\colon  H_1(X(K_i))\to\Z$.
Note that in the proof of Lemma \ref{mainlem} we showed that $\ker(\psi_i)=\ker(\phi_i\circ p_i)$.
It thus follows from Proposition \ref{prop:l2torsion} that
\[  \tau(Y_i,\psi_i)= \tau(Y_i,\phi_i\circ p_i)=n_i\cdot \tau(X(K_i),\phi_i) =n_i\cdot \tau(K_i).\]
The first equality now follows from Proposition \ref{prop:l2torsion} (3).

We now turn to the von Neumann $\rho$--invariant.
Note that $Y_i$  has one boundary component and  that
$\ker\{H_1(\partial Y_i)\to H_1(Y_i)\}\cong \Z$. Therefore, up to sign, there exists a unique
generator $l$ of $\ker\{H_1(\partial Y_i)\to H_1(Y_i)\}$. By doing Dehn filling with the slope $l$ we
get a closed 3--manifold $M_i$. Note that $\varphi\colon  Y_1\to Y_2$ canonically extends to a diffeomorphism $\varphi\colon  M_1\to M_2$.
Also note that we can identify $M_i$ with $N(K_i)_{n_i}$.

Clearly we have $H_1({M_i})=H_1(Y_i)$, in particular $b_1(M_i)=b_1(Y_i)=1$.
 We denote the projection maps $M_i\to N(K_i)$
by $p_i$ and for $i=1,2$ we pick epimorphisms ${\psi}_i\colon  H_1(M_i)\to\Z$ and $\phi_i\colon  H_1(N(K_i))\to\Z$.
The argument of  the proof of Lemma \ref{mainlem} shows that $\ker(\psi_i)=\ker(\phi_i\circ p_i)$.
 By Proposition \ref{prop:vonneumann} we
have
 \[ \ba{rcl} \rho(M_i,\psi_i) &=&\rho(M_i,\phi_i\circ p_i)\\[0.1cm]
 &=&n_i\big(\rho(N(K_i),\phi_i)\,\,-\,\,\rho(N(K_i),\pi_1(N(K_i))\to \Z\to \Z/n_i)\big).\ea \]
The second equality now follows from Proposition \ref{prop:l2torsion} (3) and (4).

\end{proof}

\subsection{Classical approaches to Theorem \ref{mainthm}}\label{section:classical}

Theorem \ref{mainthm} can be viewed as a statement about classical invariants
(i.e. invariants which are determined by the Seifert matrix) of cyclically commensurable knots.
It is a natural question whether Theorem \ref{mainthm} can be proved without using $L^2$--invariants.

We will give alternative proofs for the statement in Theorem \ref{mainthm} regarding $\tau(K_1)$ and $\tau(K_2)$.
We will also propose another approach for proving the relationship between $\rho(K_1)$ and $\rho(K_2)$, but we will not pursue this approach.

In the following let $K_1$ and $K_2$ be two oriented cyclically commensurable knots such that there exists a diffeomorphism $\varphi\colon X(K_1)_{n_1}\to X(K_2)_{n_2}$ with $b_1(X(K_1)_{n_1})=1$. We write $\eps=1$ if $\varphi$ is orientation preserving, otherwise we write $\eps=-1$.
We now identify $X(K_1)_{n_1}$ with $X(K_2)_{n_2}$ and we call this space $Y$.
We will show that
\[ n_1\cdot \tau(K_1)=n_2\cdot \tau(K_2) \]
in two different `classical' ways, using the homology growth function and by studying the orders of the Alexander modules.

\subsubsection{The homology growth function}
Given $n\in \N$ we denote by $Y_n$ the $n$--fold cyclic cover corresponding to the map $\pi_1(Y)\to \Z\to \Z/n$, where the first map is one of the two canonical epimorphisms.
It then follows from  \cite[Theorem~2.1]{SW02} that
\[ \ba{rcl}
n_1\cdot m(\Delta_K(t_1))&=&n_1\cdot \lim_{k\to \infty}\frac{1}{k}\log |\op{Tor} H_1(X(K_1)_k)|\\[1mm]
&=& \lim_{k\to \infty}\frac{1}{k}\log |\op{Tor} H_1(X(K_1)_{n_1k})|\\[1mm]
&=& \lim_{n\to \infty}\frac{1}{k}\log |\op{Tor} H_1(Y_k)|\\[1mm]
&=& \lim_{k\to \infty}\frac{1}{k}\log |\op{Tor} H_1(X(K_2)_{n_2k})|\\[1mm]
&=&n_2\cdot \lim_{k\to \infty}\frac{1}{k}\log |\op{Tor} H_1(X(K_2)_k)|=n_2\cdot m(\Delta_K(t_2)).\ea \]

\subsubsection{The Alexander polynomial and powers of variables}\label{section:powers}

Let $H$ be a torsion $\zt$-module which admits a square presentation matrix,
i.e. $H\cong \zt^l/A\zt^l$ for some matrix $A$ over $\zt$.
 We denote by $\ord_{\zt}(H):=\det_{\zt}(A)$ the order of $H$ as a $\zt$--module.
(We refer to \cite{Hi02} and \cite{Tu01} for background on orders of $\zt$--modules.)
We  factor $\ord_{\zt}(H)$  as follows:
\[ \ord_{\zt}(H)=C\cdot \prod_{i=1}^m (t-r_i),\]
where $C\in \Z$ and $r_1,\dots,r_m\in \C^*$.

Now  let $n\in \Z$. We then  view $H$ as a $\Z[s^{\pm 1}]$-module where $s$ acts on $H$
by multiplication by $t^n$. Then it follows from \cite[Lemma~4.1]{SW06} that
\be \label{ord:tn} \ord_{\Z[s^{\pm 1}]}(H) =C^n\cdot \prod_{i=1}^m(t-r_i^n).\ee
It follows immediately from Jensen's formula (\ref{equ:jensen}) that
\[ m( \ord_{\Z[s^{\pm 1}]}(H))=n\cdot m( \ord_{\Z[t^{\pm 1}]}(H)).\]

We now return to our setup of cyclically commensurable knots.
Note that $X(K_1)$, $X(K_2)$ and  $Y$ have the same infinite cyclic cover, which we denote by $\wti{Y}$.
For $i=1,2$ we now identify $H_1(X(K_i))$ with the infinite cyclic group $\ll t_i\rr$.
Furthermore we denote by $s$ the generator of $H_1(Y)$ which corresponds to $t_1^{n_1}$. Note that $s=t_2^{\eps n_2}$.
Finally we write
\[ H:=H_1(\wti{Y};\Z),\]
which we view as a module over $\Z[t_1^{\pm 1}]$, $\Z[t_2^{\pm 1}]$ respectively $\Z[s^{\pm 1}]$ in the obvious ways.
It then follows from the above discussion that
\[ \ba{rcl} n_1\cdot m(\Delta_K(t_1))&=&n_1\cdot m(\ord_{\Z[t_1^{\pm 1}]}(H))\\
&=&m(\ord_{\Z[s^{\pm 1}]}(H))\\
&=&n_2\cdot m(\ord_{\Z[t_2^{\pm 1}]}(H))=n_2\cdot m(\Delta_K(t_2)).\ea \]
From the above argument we obtain in fact the following proposition, which gives considerably more information
on cyclically commensurable knots than one can obtain from the Mahler measure.

\begin{proposition}\label{prop:alexcycliccommensurable}
Let  $K_1$ and $K_2$ be two oriented cyclically commensurable knots such that there exists a diffeomorphism $\varphi\colon X(K_1)_{n_1}\to X(K_2)_{n_2}$
with $b_1(X(K_1)_{n_1})=1$. We write
\[ \Delta_{K_1}(t)=C\cdot \prod_{i=1}^m (t-r_i) \mbox{ and }\Delta_{K_2}(t)=D\cdot \prod_{i=1}^m (t-s_i)\]
where $C,D\in \Z$ and $r_1,\dots,r_m,s_1,\dots,s_m\in \C^*$. Then $C^{n_1}=D^{n_2}$ and
\[ \{r_1^{n_1},\dots,r_m^{n_1}\}=\{s_1^{n_2},\dots,s_m^{n_2}\}\]
as sets with multiplicities.
\end{proposition}

\subsubsection{The Blanchfield pairing of cyclically commensurable knots}

The statement and the proof of Proposition \ref{prop:alexcycliccommensurable} shows that more information regarding cyclically commensurable knots
can be obtained by studying the Alexander polynomial directly.

Recall that to a knot $K$ we can associate the
Blanchfield pairing (see  \cite{Bl57,Ke75a,Hi02})
\[ \l(K)\colon  H_1(X(K);\zt)\times H_1(X(K);\zt)\to \Q(t)/\zt,\]
which is a hermitian, non-singular pairing on the Alexander module.
Note that the  Blanchfield pairing completely determines the Levine-Tristram signatures,
we refer to \cite{Ke75a,Ke75b} and \cite{Le89} for details.
We expect that a careful study of the Blanchfield pairing can be used to reprove the second statement of Theorem \ref{mainthm}
and to extract considerably more information regarding the Levine-Tristram signatures of cyclically commensurable knots.

\section{Other $L^2$--invariants}\label{section:otherl2} \label{section:morel2}

The approach taken in the proof of Theorem \ref{mainthm} can be formalized as follows.
Let $\NN$ be a class of 3--manifolds (possibly equipped with a Riemannian structure or an orientation)
let $\GG$ be an abelian group and let $\II$ be a $\GG$--valued invariant of pairs $(N,\a)$ where $N\in \NN$ and $\a\colon  \pi_1(N)\to G$ is a group homomorphism. Suppose $\II$ has the following two properties:
\bn
\item[($\AA 1$)] Let $N\in \NN$ and let  $\a\colon  \pi_1(N)\to G$ be a homomorphism to a group and let  $\b\colon  G\to H$ be a monomorphism of groups, then
\[ \II(N,\a\colon  \pi_1(N)\to G)=\II(N,\b\circ \a\colon  \pi_1(N)\to H) \quad \mbox{ (Induction)}.\]
\item[($\AA 2$)]  Let $N\in \NN$, let $\a\colon  \pi_1(N)\to G$  be a homomorphism to a group
 and let $\varphi\colon  \pi_1(N)\to A$ be an epimorphism to a finite group $A$ which factors through $\a$. Denote the corresponding cover by $N_A$. If $N_A\in \NN$, then
\[ \II(N_A,\pi_1(N_A)\to \pi_1(N)\xrightarrow{\a}\G )=|A|\cdot \II(N,\a) \quad \mbox{ (Restriction)}.\]
\en
If $\NN$ contains finite cyclic covers of knot complements, then
the argument of the proof of Theorem \ref{mainthm2v2} shows that if $K_1$ and $K_2$ are two admissible knots such that there exists a diffeomorphism $X(K_1)_{n_1}\to X(K_2)_{n_2}$ (which respects any extra structure of $\NN$), then
\[ n_1\cdot \II(X(K_1),\phi_{K_1})=n_2\cdot \II(X(K_2),\phi_{K_2}).\]
(An analogous statement also holds for $\II(N(K),\phi_{K})$ if $\NN$ contains 0--framed surgeries.)
A straightforward argument also shows that if $N_1,N_2\in \NN$ admit a diffeomorphism
(which respects any extra structure of $\NN$) between an $n_1$--fold cover of $N_1$ and an $n_2$--fold cover of $N_2$, then
\[ n_1\cdot \II(N_1,\id\colon  \pi_1(N_1)\to \pi_1(N_1))= n_2\cdot \II(N_2,\id\colon  \pi_1(N_2)\to \pi_1(N_2)).\]
Put differently, such an invariant $\II$ gives rise to invariants of cyclic commensurability of knots and to invariants of commensurability of 3--manifolds.

The following table lists invariants which satisfy the above `axioms' ($\AA 1$) and ($\AA 2$):
\[
\ba{|l|l|l|l}
\hline
\NN & \GG &\II \\
\hline
\mbox{3--manifolds} & \R &\mbox{$L^2$--torsion} \\
\mbox{closed oriented 3--manifolds} & \R/\Q &\mbox{von Neumann $\rho$--invariant}\\
\mbox{3--manifolds}& \R& \mbox{$L^2$--Betti numbers}\\
\mbox{closed, oriented, hyperbolic 3--manifolds}&\R &\mbox{von Neumann $\eta$--invariant}\\
\hline \ea
\]
The fact that the first two invariants have the required properties is basically the content of
Propositions \ref{prop:l2torsion} and \ref{prop:vonneumann}. The only missing ingredient is the well--known fact that given any 3--manifold and any epimorphism $\a\colon  \pi_1(N)\to G$ to a finite group we have $\rho(N,\a)\in \Q$.
(This in turn is a well--known consequence of the fact that $\Omega_3(G)$ is finite and that $\rho(N,\a)$ can thus be computed in terms of signatures of bounding 4--manifolds, see e.g. \cite{CW03} for details.)
The $L^2$--torsion of 3--manifolds thus gives rise to a commensurability invariant, but this invariant is well--known: L\"uck and Schick  \cite{LS99} showed that the invariant $-6\pi\cdot \tau(N,\id\colon  \pi_1(N)\to \pi_1(N))$
equals the sum of the volumes of the hyperbolic pieces in the JSJ decomposition of $N$.
On the other hand $\rho(N,\id\colon  \pi_1(N)\to \pi_1(N))\in \R/\Q$ gives rise to a new commensurability invariant of oriented 3--manifolds, unfortunately there are no known methods for computing it, let alone for showing that it is non--trivial in $\R/\Q$.

For the fact that $L^2$--Betti numbers have the desired properties we refer to \cite[Theorem~6.54~(6)~and~(7)]{Lu02}, but unfortunately by
\cite[Theorems~4.1~and~4.2]{Lu02} they do not give rise to interesting commensurability invariants.

Finally the fact that the von Neumann $\eta$--invariant of closed, oriented hyperbolic 3--manifolds has the desired properties follows immediately from the proof of Theorem \ref{prop:vonneumann}.
Unfortunately we do not know any methods for computing this invariant. We do not even know how to calculate $\eta(N(K),\phi_K)$ (assuming that $N(K)$ is hyperbolic).

\section{Examples}\label{section:examples}

\subsection{Examples of commensurable knots which are not cyclically commensurable}\label{section:ar92}
In \cite{AR92} Aitchison and Rubinstein introduce the two dodecahedral knots $D_s$ and $D_f$. They are commensurable since the complements cover the same orbifold.
Note that $D_f$ is fibered whereas $D_s$ is not. Furthermore $\mbox{genus}(D_s)=4$
and $\mbox{genus}(D_f)=5$. Also note
that
\[ \ba{rcl}
\Delta_{D_s}(t)\hspace{-0.2cm}&=&\hspace{-0.2cm}25(t^{-4}\hspace{-0.12cm}+\hspace{-0.12cm}t^4)-250(t^{-3}\hspace{-0.12cm}+\hspace{-0.12cm}t^3)t+1035(t^{-2}\hspace{-0.12cm}+\hspace{-0.12cm}t^2)-2300(t^{-1}\hspace{-0.12cm}+\hspace{-0.12cm}t)+2981t^4\\
\Delta_{D_f}(t)\hspace{-0.2cm}&=&\hspace{-0.2cm}-(t^{-5}\hspace{-0.12cm}+\hspace{-0.12cm}t^5)+29(t^{-4}\hspace{-0.12cm}+\hspace{-0.12cm}t^4)-254(t^{-3}\hspace{-0.12cm}+\hspace{-0.12cm}t^3)+1035(t^{-2}\hspace{-0.12cm}+\hspace{-0.12cm}t^2)-2304(t^{-1}\hspace{-0.12cm}+\hspace{-0.12cm}t)+2991\ea
\]
Finally note that the two knots are admissible, hence every condition of Lemma \ref{mainlem} shows that $D_s$ and $D_f$ are not cyclically commensurable.
(This was of course already known to Aitchison and Rubinstein.)
At the moment this  is the only known example of commensurable but not cyclically commensurable knots.


\subsection{Examples of cyclically commensurable knots}\label{section:example2}

Consider the knots $K_1=9_{48}$ and $K_2=12^n_{642}$.
They have the property that $X(K_1)_8=X(K_2)_6$ (see \cite[Section~5.1]{RW08}).
First note that $\genus(K_1)=\genus(K_2)=2$, that both are fibered  and that
\[ \ba{rcl} \Delta_{K_1}(t)&=&t^4-7t^3+11t^2-7t+1,\\
\Delta_{K_2}(t)&=&t^4+7t^3-15t^2+7t+1.\ea \]
In particular they are admissible. Note that all the above invariants are in line with Lemma \ref{mainlem}.


Using the substitution $x=t+t^{-1}$ one can easily compute the zeros of $\Delta_{K_1}(t)$, they are
\[ t_{1,\pm}^*=\frac{1}{2}\left(x_1^*\pm \sqrt{(x_1^*)^2-4}\right), \mbox{ with } x_1^*=\frac{1}{2}\left(7\pm \sqrt{13}\right)
\mbox{ and } *\in \{+,-\}.\]
Two of the zeros lie on the unit circle and only one zero lies outside of the unit circle, we therefore get
\[ m(\Delta_{K_1}(t))=\frac{1}{4}\left(7+\sqrt{13}+\sqrt{46+14\sqrt{13}}\right)\approx 5.10696....\]
We can also compute
 the zeros of $\Delta_{K_2}(t)$, they are
\[ t_{2,\pm}^*=\frac{1}{2}\left(x_2^*\pm \sqrt{(x_2^*)^2-4}\right), \mbox{ with } x_2^*=\frac{1}{2}\left(-7\pm 3\sqrt{13}\right)
\mbox{ and } *\in \{+,-\}.\]
It follows that
\[ m(\Delta_{K_2}(t))=\frac{1}{4}\left| -7-3\sqrt{13}-\sqrt{150+42\sqrt{13}}\right|\approx 8.79462....\]
A somewhat painful calculation will show that
\[ |t_{1+}^+|^6=|t_{2-}^-|^8,\]
which by (\ref{equ:jensen}) then implies that
\[ 6 \,\tau(K_2)=6 \ln (m(\Delta_{K_2})(t))=8 \ln (m(\Delta_{K_1})(t))=8 \,\tau(K_1), \]
in agreement with Theorem \ref{mainthm2v2}.
Note that it also follows immediately from Theorem \ref{mainthm2v2} that if $n_1$ and $n_2\in \N$ are such that $X(K_1)_{n_1}\cong X(K_2)_{n_2}$, then
\be \label{equ:43} \frac{n_1}{n_2}=\frac{\tau(K_2)}{\tau(K_1)}=\frac{4}{3}.\ee

We now turn to the study of the von Neumann $\rho$--invariants of $K_1$ and $K_2$. From the above computation we see that there are exactly two zeros of $\Delta_{K_1}(t)$ on the unit circle, namely
\[ t_{1,\pm}=\frac{1}{2}\left(x_1\pm i \sqrt{|x_1^2-4|}\right), \mbox{ with }x_1=\frac{1}{2}\left(7-\sqrt{13}\right).\]
Using $\s(K_1)=2$ and Proposition  \ref{propu1sign}  we get
\[ \rho(K_1)=\int_{w\in S^1} \s(K_1,w)=\s(K_1)\left(1-\frac{\cos^{-1}(x_1/2)}{\pi}\right)\approx 1.645123....\]
Similarly we
see that there are exactly two zeros of $\Delta_{K_2}(t)$ on the unit circle, namely
 \[ t_{2,\pm}=\frac{1}{2}\left(x_2\pm i \sqrt{|x_2^2-4|}\right), \mbox{ with }x_2=\frac{1}{2}\left(-7+3\sqrt{13}\right).\]
Using $\s(K_2)=2$ we compute
\[ \rho(K_2)=\int_{w\in S^1} \s(K_2,w)=\s(K_2)\left(1-\frac{\cos^{-1}(x_2/2)}{\pi}\right)\approx 1.806503....\]
Furthermore using again Proposition  \ref{propu1sign}  we compute $\sum_{k=1}^4 \s(K_1,e^{2\pi i k/4})=(4-1)\cdot 2=6$
and $\sum_{k=1}^3 \s(K_2,e^{2\pi i k/6})=(3-1)\cdot 2=4$.
Accordingly we compute
\[ \ba{rclcl} 4\rho(K_1)-\sum_{k=1}^4 \s(K_1,e^{2\pi ik/4}) &\approx &4\cdot 1.645123-6&=&0.580492\\[0.1cm]
3\rho(K_2)-\sum_{k=1}^3 \s(K_2,e^{2\pi ik/3})&\approx&3\cdot 1.806503-4&=&1.419509,
\ea \]
This shows, that $X(K_1)_4$ and $X(K_2)_3$ are not diffeomorphic. Note that this could not be detected using the $\tau$--invariant.

We also compute $\sum_{k=1}^8 \s(K_1,e^{2\pi i k/8})=14$ and  $\sum_{k=1}^6 \s(K_2,e^{2\pi i k/6})=10$.
The manifolds $X(K_1)_{8}$ and $X(K_2)_6$ are diffeomorphic.
 Accordingly we compute
\[ \ba{rclcl} 8\rho(K_1)-\sum_{k=1}^8 \s(K_1,e^{2\pi ik/8}) &\approx &8\cdot 1.645123-14&=&-0.839016\\[0.1cm]
6\rho(K_2)-\sum_{k=1}^6 \s(K_2,e^{2\pi ik/6})&\approx&6\cdot 1.806503-10&=&0.839018,
\ea \]
in line with Theorem \ref{mainthm2v2}. Note that this shows that $X(K_1)_{8}$ and $X(K_2)_6$ are orientation--reversing diffeomorphic.

We conclude the discussion of this example with the following lemma.

\begin{lemma}
The knots $K_1=9_{48}$ and $K_2=12^n_{642}$ are not orientation--preserving commensurable.
\end{lemma}

\begin{proof}
Suppose there exist $n_1$ and $n_2$ such that there exists an orientation--preserving diffeomorphism $X(K_1)_{n_1}\to X(K_2)_{n_2}$.
By (\ref{equ:43}) there exists $n\in \N$ such that $n_1=4n$ and $n_2=3n$. It follows from Theorem \ref{mainthm2v2} that
\[ n(4\rho(K_1)-3\rho(K_2))\in \Z.\]
In particular
\[ \left(e^{2\pi i(4\rho(K_1)-3\rho(K_2))}\right)^n=1.\]
We write $t_i:=t_{i+}$. Then
\[ e^{2\pi i\rho(K_1)}=t_{1}^2 \mbox{ and } e^{2\pi i\rho(K_2)}=t_{2}^2.\]
It thus follows from the above that $t_{1}^4t_2^{-3}$ is a root of unity.

We will show that this is not the case. First note that $[\Q(t_{1}^4t_2^{-3}):\Q]$ clearly divides $8$. In particular if $t_{1}^4t_2^{-3}$ is an $n$-th root of unity,
then $\varphi(n)$ has to divide $8$. The only possibilities are $n=2,3,5,8,12,15,16,24$.
Note that
\[ \ba{rclclclclcl} t_1&=&e^{i\varphi_1},& \mbox{ with }&\varphi_1&=&\arccos(x_1/2)&\approx& 0.557439979\\
t_2&=&e^{i\varphi_2},& \mbox{ with }&\varphi_2&=&\arccos(x_2/2)&\approx& 0.303944246.\ea \]
But $480\cdot (4\cdot 0.557439979-3\cdot 0.303944246)$ is not a multiple of $2\pi$.
This concludes the proof of the lemma.
\end{proof}

%
%

\subsection{Examples of non--cyclically commensurable knots}
Let $K$ be an admissible knot with $\rho(K)$ irrational.
(For example $K$ could be any knot with  $\Delta_K(t)=a+(1-2a)^t+at^2$ and $a> 1$.)
Then it follows from Theorem \ref{mainthm2v2} that $K$ is not cyclically commensurable to any knot $J$ with $\rho(J)$ rational. (For example any knot which is slice has the property that $\rho(J)=0$, see \cite{COT04}).
Also note that such knot $K$ is not cyclically orientation--preserving commensurable to $-K$.

\end{document}